\documentclass[11pt,reqno]{article}

\usepackage{amsmath, amssymb, amsthm} 
\usepackage{empheq}
\usepackage{mathpazo}
\usepackage{nameref}
\usepackage[margin=3cm]{geometry}
\usepackage[titletoc]{appendix}

\usepackage{caption,subcaption}
\usepackage[doc]{optional}
\usepackage{graphicx, xcolor, soul}
\usepackage{float}
\restylefloat{table*}
\usepackage{booktabs} 

\definecolor{labelkey}{rgb}{0,0.08,0.45}
\definecolor{refkey}{rgb}{0,0.6,0.0}
\definecolor{Brown}{rgb}{0.45,0.0,0.05}
\definecolor{lime}{rgb}{0.00,0.8,0.0}
\definecolor{lblue}{rgb}{0.5,0.5,0.99}

\usepackage{hyperref}

\usepackage{sidecap}
\definecolor{myblue}{rgb}{.9, .9, 1} 
  \newcommand*\mybluebox[1]{%
    \colorbox{myblue}{\hspace{1em}#1\hspace{1em}}}

\newcommand{\sepp}{\setlength{\itemsep}{-2pt}}

\newcommand{\menge}[2]{\left\{{#1}~\big |~{#2}\right\}}

\newcommand{\scal}[2]{\left\langle {#1},{#2} \right\rangle}

\newcommand{\NN}{\ensuremath{\mathbb N}}
\newcommand{\nnn}{\ensuremath{{n\in{\mathbb N}}}}

\newcommand{\RR}{\ensuremath{\mathbb R}}

\newcommand{\ba}{\ensuremath{\mathbf{a}}}
\newcommand{\bb}{\ensuremath{\mathbf{b}}}

\newcommand{\bx}{\ensuremath{\mathbf{x}}}

\newcommand{\bg}{\ensuremath{\mathbf{g}}}

\newcommand{\bA}{\ensuremath{{\mathbf{A}}}}
\newcommand{\bB}{\ensuremath{{\mathbf{B}}}}
\newcommand{\bE}{\ensuremath{{\mathbf{E}}}}

\newcommand{\bT}{\ensuremath{{\mathbf{T}}}}
\newcommand{\bX}{\ensuremath{{\mathbf{X}}}}

\newcommand{\TAB}{T_{(A,B)}}

\newcommand{\epi}{\ensuremath{\operatorname{epi}}}

\newcommand{\Id}{\ensuremath{\operatorname{Id}}}

\usepackage[capitalize]{cleveref}
\crefname{equation}{}{equations}
\crefname{chapter}{Appendix}{chapters}
\crefname{item}{}{items}
\crefname{figure}{Figure}{figures}

\makeatletter
\def\th@plain{%
  \thm@notefont{}%
  \itshape %
}
\def\th@definition{%
  \thm@notefont{}%
  \normalfont %
}
\makeatother

\newtheorem{theorem}{Theorem}[section]
\newtheorem{thm}[theorem]{Theorem}
\newtheorem{lemma}[theorem]{Lemma}
\newtheorem{lem}[theorem]{Lemma}
\newtheorem{corollary}[theorem]{Corollary}

\theoremstyle{definition}
\newtheorem{example}[theorem]{Example}
\newtheorem{ex}[theorem]{Example}

\theoremstyle{definition}
\newtheorem{remark}[theorem]{Remark}

\providecommand{\abs}[1]{\lvert#1\rvert}

\providecommand{\norm}[1]{\lVert#1\rVert}
\providecommand{\normsq}[1]{\lVert#1\rVert^2}
\providecommand{\bk}[1]{\left(#1\right)}
\providecommand{\stb}[1]{\left\{#1\right\}}
\providecommand{\innp}[1]{\langle#1\rangle}

\providecommand{\RA}{\Rightarrow}

\providecommand{\RR}{\mathbb{R}}

\newcommand{\fix}{\ensuremath{\operatorname{Fix}}}

\providecommand{\epi}{\operatorname{epi}}

\providecommand{\Id}{\operatorname{{ Id}}}

\providecommand{\fady}{\varnothing}

\providecommand{\NN}{\mathbb{N}}

\providecommand{\fix}{\operatorname{Fix}}

\providecommand{\Id}{\operatorname{Id}}

\providecommand{\fady}{\varnothing}

\providecommand{\RR}{\mathbb{R}}
\providecommand{\NN}{\mathbb{N}}

\makeatletter
\def\namedlabel#1#2{\begingroup
   \def\@currentlabel{#2}%
   \label{#1}\endgroup
}
\makeatother

\allowdisplaybreaks

\begin{document}

\title{The Douglas--Rachford algorithm\\ in the affine-convex case}

\author{
Heinz H.\ Bauschke\thanks{
Mathematics, University of British Columbia, Kelowna, B.C.\ V1V~1V7, Canada. 
E-mail: \texttt{heinz.bauschke@ubc.ca}. The research of all
authors was supported by NSERC and by the Canada Research Chair
program.},~
Minh N.\ Dao\thanks{
 Mathematics, University of British Columbia, Kelowna, B.C.\ V1V~1V7, Canada,
and
Department of Mathematics and Informatics,
Hanoi National University of Education, 136 Xuan Thuy, Hanoi, Vietnam.
E-mail: \texttt{minhdn@hnue.edu.vn}.}, 
~and Walaa M.\ Moursi\thanks{
Mathematics, University of British Columbia, Kelowna, B.C.\ V1V~1V7, Canada,
and 
Mansoura University, Faculty of Science, Mathematics Department, 
Mansoura 35516, Egypt. 
E-mail: \texttt{walaa.moursi@ubc.ca}.}
}

\date{May 23, 2015}

\maketitle

\begin{abstract}
\noindent
The Douglas--Rachford algorithm is a simple yet effective
method for solving convex feasibility problems. However, if
the underlying constraints are inconsistent, then the 
convergence theory is incomplete. 
We provide convergence results when one constraint is an
affine subspace. As a consequence, we extend a result by Spingarn
from halfspaces to general closed convex sets admitting least-squares solutions. 
\end{abstract}
{\small

}

\section{Introduction}

We shall assume throughout this paper that
$X$ is a real Hilbert space 
with inner product $\innp{\cdot,\cdot}$ and induced norm $\norm{\cdot}$, 
and that 
\begin{empheq}[box=\mybluebox]{equation}
A \text{~and~} B \text{~~are nonempty closed convex (not
necessarily intersecting) subsets of~~} X.
\end{empheq} 
Consider the problem of finding a best approximation pair relative to $A$ and $B$ (see \cite{BCL04}, \cite{Luk08}), 
that is to 
\begin{empheq}[box=\mybluebox]{equation}
\text{find~~} (a, b) \in A\times B \text{~~such that~~} \|a - b\| =\inf\|A - B\|.
\end{empheq}
Recall that the Douglas--Rachford splitting operator \cite{LM79} 
for the ordered pair of sets $(A, B)$ is defined by
\begin{empheq}[box=\mybluebox]{equation}
\label{def:T}
T =\TAB :=\tfrac{1}{2}(\Id +R_BR_A) =\Id -P_A +P_B R_A,
\end{empheq} 
where $P_A$ is the projector onto $A$
and $R_A :=2P_A-\Id$ is the reflector onto $A$.
Let $x \in X$. In the consistent case, 
when $Z_{A,B}:=A\cap B\neq \fady $,
the \emph{``governing sequence"} $(T^n x)_{n\in \NN}$
generated by iterating
 the Douglas--Rachford operator 
converges weakly to a fixed point\footnote{$\fix T 
=\menge{x\in X}{x=Tx}$ is the set of fixed points of $T$.}
of $T$ (see \cite{LM79}),
and the \emph{``shadow sequence"} $(P_AT^n x)_{n\in \NN}$ 
converges weakly to a point in $A\cap B$ (see \cite{Sva11} or
\cite[Theorem~25.6]{BC01}). 
For further information on the Douglas--Rachford algorithm (DRA), 
see also \cite{LM79} and \cite{EB92}.

In \cite{BCL04}, the authors showed that 
in the inconsistent case, when $A\cap B= \fady$, 
$(P_AT^n x)_{n\in \NN}$ remains bounded with 
the weak cluster points
of $(P_AT^nx, P_BP_AT^nx)_\nnn$
 being best approximation 
pairs relative to $A$ and $B$ 
whenever $g:=P_{\overline{B -A}}0 \in B -A$.
\emph{The goal of this paper is to study the case 
 when $A\cap B$ is possibly empty in the setting that
one of the sets $A$ and $B$ is a closed affine subspace 
of $X$.}
Our results show that the shadow sequence will always converge  
to a best approximation solution in $A\cap (B -g)$. As a consequence 
 we obtain a far-reaching refinement of 
Spingarn's splitting method introduced in \cite{Spi87}.

\section{Main results}
We start with the following key lemma, which is well known when
$A=X$. 

\begin{lem}\label{Great:prop}
Let $A$ be a closed linear subspace of $X$, let $C$ 
be a nonempty closed convex subset of $A$,
and
let $(x_n)_\nnn$ be a sequence in $X$.
Suppose that $(x_n)_\nnn$ is Fej\'{e}r monotone with respect to 
$C$, i.e., $(\forall \nnn)$ $(\forall c\in C)$ 
$\norm{x_{n+1}-c}\le \norm{x_{n}-c}$, 
and that all its weak cluster points of $(P_Ax_n)_\nnn$
lie in $C$. Then 
$
(P_A x_n)_\nnn
$ converges weakly to some point in $C$.
 \end{lem}
 
\begin{proof}
 Since $(x_n)_\nnn$ is bounded (by e.g., \cite[Proposition~5.4(i)]{BC01})
  and $P_A$ is (firmly) nonexpansive 
  we learn that $(P_Ax_n)_\nnn$ is bounded
  and by assumption, its weak cluster points lie in $C\subseteq A$. 
Now let $c_1$ and $c_2$ be in $C$. On the one hand the
Fej\'{e}r monotonicity of $(x_n)_\nnn$  implies the convergence of the 
 sequences
 $(\normsq{x_n-c_1})_\nnn$ and $(\normsq{x_n-c_2})_\nnn$
 by e.g., \cite[Proposition~5.4(ii)]{BC01}.
On the other hand, expanding and simplifying yield
$
\normsq{x_n-c_1}-\normsq{x_n-c_2}
=\normsq{x_n}+\normsq{c_1}-2\innp{x_n,c_1}
-\normsq{x_n}-\normsq{c_2}+2\innp{x_n,c_2}
=\normsq{c_1}-2\innp{x_n,c_1-c_2}-\normsq{c_2},
$
which in turn implies that  
$(\innp{x_n,c_1-c_2})_\nnn$ converges.
Since $c_1\in A$ and $c_2\in A$ we have
\begin{equation}\label{eq:gold}
\innp{x_n,c_1-c_2}=\innp{x_n,P_Ac_1-P_Ac_2}=\innp{x_n,P_A(c_1-c_2)}
=\innp{P_A x_n,c_1-c_2}. 
\end{equation}
Now assume that $(P_A x_{k_n})_\nnn$ and $(P_A x_{l_n})_\nnn$
are subsequences of $(P_A x_n)_\nnn$
such that $P_A x_{k_n} \rightharpoonup c_1$
and 
$P_A x_{l_n} \rightharpoonup c_2$.
By the uniqueness of the limit in \cref{eq:gold}
we conclude that $\innp{c_1,c_1-c_2}=\innp{c_2,c_1-c_2}$
or equivalently $\normsq{c_1-c_2}=0$, hence 
$(P_A x_{n})_\nnn$ has a unique weak cluster point 
which completes the proof.
 \end{proof}

From now on we work under the assumption that
\begin{empheq}[box=\mybluebox]{equation}
\label{def:g}
g=g_{(A,B)} :=P_{\overline{B -A}}0 \in B -A.
\end{empheq} 
In view of \cref{def:g} we have
\begin{empheq}[box=\mybluebox]{equation}
\label{defn:E}
E=E_{(A,B)} :=A\cap (B -g) \neq\fady \quad\text{and}\quad 
F =F_{(A,B)}:=(A +g)\cap B \neq\fady.
\end{empheq} 
For sufficient conditions
on when $g\in B-A$ 
(or equivalently the sets $E$ and $F$ are nonempty) 
we refer the reader to \cite[Facts~5.1]{BB94}.
\begin{lemma}
\label{l:aux}
Let $x\in X$. Then the following hold:
\begin{enumerate}
\item
\label{l:aux_gap}
If $C\in\stb{A,B}$ is a closed affine subspace of $X$, 
 then $g\in (C-C)^\perp$.
\item
\label{l:aux_Fejer}
The sequence $(T^n x-ng)_\nnn$
is Fej\'{e}r monotone with respect to $E$. 
\item
\label{l:aux_PA}
The sequence $(P_AT^n x)_\nnn$ is bounded
and its weak cluster points lie in $E$.
\item
\label{l:aux_PB}
If $B$ is a closed affine subspace, 
then $P_BT^n x -P_AT^n x \to g$, 
the sequence $(P_BT^n x)_\nnn$ is bounded
and all weak cluster points lie in $F$.
\item
\label{l:aux_single}
If $E=\stb{\bar x}$ and hence $F=\stb{\bar x +g}$, 
then
$P_AT^n x\rightharpoonup \bar x$ and $P_BT^n x\rightharpoonup \bar x +g$. 
\end{enumerate}
\end{lemma}
\begin{proof}
\ref{l:aux_gap}: See \cite[Corollary~2.7 and Remark~2.8(ii)]{BCL04}.
\ref{l:aux_Fejer}: 
It follows\footnote{
We use $N_C$ to denote the \emph{normal cone} operator
associated with a nonempty closed convex subset $C$ of $X$.
}
 from 
\cite[Theorem~3.5]{BCL04}
that $E+N_{\overline{A-B}}(-g)\subseteq 
\fix (-g+T) :=\menge{x \in X}{x =-g +Tx} 
\subseteq -g+E+N_{\overline{A-B}}(-g)$.
Consequently, $E\subseteq \fix (-g+T) $.
Moreover, \cite[Remark~3.15]{BCL04}
implies that the sequence $(T^n x-ng)_\nnn$
is Fej\'{e}r monotone with respect to $\fix (-g+T)$.
\ref{l:aux_PA}: See \cite[Theorem~3.13(iii)(b)]{BCL04}.
\ref{l:aux_PB}: See \cite[Theorem~3.17]{BCL04}.
\ref{l:aux_single}: This follows from \ref{l:aux_PA} and \ref{l:aux_PB}.
\end{proof}

We are now ready for our main results.
\begin{thm}[convergence of DRA 
when $A$ is a closed affine subspace]
\label{t:affcvx}
Suppose that $A$ is a closed affine subspace of $X$,
and let $x\in X$.  
Then the following hold:
\begin{enumerate}
\item\label{t:affcvx:i}
The \emph{shadow sequence} 
$(P_A T^n x)_\nnn$ 
converges weakly to some  
point in
$E=A\cap (B -g)$. 
\item\label{t:affcvx:ii}
No general conclusion can be 
drawn about the sequence 
$(P_B T^n x)_\nnn$. 
\end{enumerate}
\end{thm}
\begin{proof}
\ref{t:affcvx:i}: 
After translating the sets $A$ and $B$ by a vector, if necessary,
we can and do assume that $A$ is a closed linear subspace
of $X$.
Using \cref{l:aux}\ref{l:aux_gap} we learn that
$(\forall n\in \NN)$ $P_AT^n x =P_A(T^n x -ng)$.
Note that $E =A\cap(B -g)\subseteq A$.
Now combine \cref{l:aux}\ref{l:aux_Fejer}--\ref{l:aux_PA} 
and \cref{Great:prop} with $C =E$, 
and $(x_n)_\nnn$ replaced by $(T^n x-ng)_\nnn$. 
\ref{t:affcvx:ii}:
In fact, $(P_B T^n x)_\nnn$ 
can be unbounded (see \cref{ex:PB:unbd}) or bounded (e.g., when
$A=B=X$). 
\end{proof}

 \begin{example}\label{ex:PB:unbd}
Suppose that $X=\RR^2$, that $A=\RR\times\stb{0}$ and
that $B=\epi\bk{\abs{\cdot}+1}$. Then $A\cap B=\fady$
and for the starting point $x\in \left [-1,1\right]\times\stb{0} $ we have
$(\forall n\in\stb{1,2,\ldots})$
$T^n x=(0,n) \in B$ 
and therefore
$\norm{P_B T^nx}=\norm{T^nx}=n \to \infty$.
\end{example} 
\begin{proof}
Let $x=(\alpha,0)$ with $\alpha\in \left[-1,1\right]$.
We proceed by induction. When $n=1$ we have
$T(\alpha,0)=P_{A^\perp}(\alpha,0)+P_BR_A(\alpha,0)
=P_B(\alpha,0)=(0,1)$.
Now suppose that for some $(n\in \stb{1,2,\ldots})$
$T^nx=(0,n)$. Then $T^{n+1}x=T(0,n)=
P_{A^\perp}(0,n)+P_BR_A(0,n)=(0,n)+P_B(0,-n)=(0,n+1)\in B$.
\end{proof}

When $B$ is an affine subspace, 
the convergence theory is even more satisfying:

\begin{thm}[convergence of DRA  
when $B$ is a closed affine subspace]
\label{t:affcvx:B}
Suppose that $B$ is a closed affine subspace of $X$, and 
let $x\in X$.
Then the following hold:
\begin{enumerate}
\item\label{t:affcvx:B:ii}
The 
\emph{shadow sequence}  $(P_A T^n x)_\nnn$
converges weakly to some  
point in 
$E=A\cap (B -g)$.
\item\label{t:affcvx:B:i}
The sequence
$(P_B T^n x)_\nnn$ 
converges weakly to some  
point in
$F=(A+g)\cap B$. 
\end{enumerate}
\end{thm}
\begin{proof}
\ref{t:affcvx:B:i}: 
Combine \cref{t:affcvx}\ref{t:affcvx:i} 
and \cite[Corollary~2.8(i)]{BM_order}.
\ref{t:affcvx:B:ii}:
Combine \ref{t:affcvx:B:i} and
 \cref{l:aux}\ref{l:aux_PB}. 
\end{proof}

It is tempting to conjecture that 
\cref{t:affcvx}\ref{t:affcvx:i} remains true when $A$ is just convex and not
necessarily a subspace. 
While this \emph{statement} may be true\footnote{In
\cite[Remark~3.14(ii)]{BCL04}, the authors claim otherwise but 
forgot to list the assumption that
$A\cap B\neq\varnothing$.}, the \emph{proof}
of \cref{t:affcvx}\ref{t:affcvx:i} does \emph{not} admit such an extension:

 \begin{example}
 Suppose that $X=\RR$, that
 $A=[1,2]$ and that $B=\stb{0}$.
 Then $g=-1$ and $E=\stb{1}$.
 Let $x = 4$. We have 
 $(T^n x)_\nnn =(4, 2, 0, -1, -2, -3, \dots)$, 
 $P_AT^n x\to 1 \in E$ and 
$(\forall n \in \{2, 3, 4, \dots\})$ $T^n x -ng =-(n -2) -n(-1) =
2 \in A$ and $P_A(T^n x- ng) = 2 \in A\smallsetminus E$.  
In the proof of \cref{t:affcvx}\ref{t:affcvx:i}, we
had $(P_AT^n x)_\nnn = (P_A(T^n x -ng))_\nnn$ which is strikingly
false here. 
\end{example}

\section{Spingarn's method}
In this section we discuss the problem to 
find \emph{least-squares solutions} of $\bigcap_{i=1}^M C_i$,
i.e., to 
\begin{empheq}[box=\mybluebox]{equation}
\label{e:moresets}
\text{find minimizers of $\sum_{i=1}^M d^2_{C_i}$,} 
\end{empheq}
where
$C_1, \ldots, C_M$ are nonempty closed convex 
(possibly nonintersecting) subsets of $X$ with corresponding
distance functions $d_{C_1},\ldots,d_{C_M}$.
Now consider the product Hilbert space $\bX :=X^M$, 
with the inner product $((x_1, \dots, x_M), (y_1, \dots, y_M)) 
\mapsto \sum_{i=1}^M \scal{x_i}{y_i}$.
We set 
\begin{equation}
\bA =\menge{(x,\dots, x) \in \bX}{x \in X} 
\quad\text{and}\quad \bB =C_1\times \cdots \times C_M.
\end{equation} 
Then the projections of $\bx =(x_1, \dots, x_M) \in \bX$ onto $\bA$ 
and $\bB$ are given by, respectively, 
$P_\bA\bx =\left( \frac{1}{M}\sum_{i=1}^M x_i, \ldots, 
\frac{1}{M}\sum_{i=1}^M x_i \right)$
and $P_\bB\bx =\left( P_{C_1}x_1, \dots, P_{C_M}x_M \right)$.
Now assume that
\begin{empheq}[box=\mybluebox]{equation}
\label{e:gap'}
\bg =(g_1, \dots, g_M) :=P_{\overline{\bB -\bA}}0 \in \bB -\bA. 
\end{empheq}
Then we have 
\begin{equation}
\label{e:rela}
\bE :=\bA\cap (\bB -\bg) \neq\varnothing, \quad\text{and}\quad
(x, \dots, x) \in \bA\cap (\bB -\bg) 
~\Leftrightarrow~ x \in \bigcap_{j=1}^M (C_j -g_j).
\end{equation}
Using \cite[Section 6]{BB94}, we see that 
the $M$-set problem \cref{e:moresets} is equivalent to the
\emph{two-set} problem
\begin{equation}
\label{e:prob'}
\text{find least-squares solutions of $\bA\cap \bB$}.
\end{equation}
It follows from \cref{e:gap'} and \cref{e:rela} that
$\bg$ is the unique vector in $\overline{\bB-\bA}$ that satisfies
\begin{equation}
\label{e:waterloo}
    \left. \begin{array}{c}
        (w_1,w_2,\ldots,w_M)\neq (g_1,g_2,\ldots,g_M),  \\
       \text{and~~} \bigcap\limits_{j=1}^M (C_j -w_j) \neq\varnothing\\
    \end{array} \right\}\;\; \RA\;\;
    \sum_{j=1}^{M}\normsq{w_j}>\sum_{j=1}^{M}\normsq{g_j}.
\end{equation}
We have the following result for the problem 
of finding a least-squares solution for the intersection of a finite family of sets.
\begin{corollary}
\label{c:moresets}
Suppose that $C_1, \dots, C_M$ are closed convex subsets of $X$.
Let $\bT =\Id -P_\bA +P_\bB R_\bA$, let $\bx \in \bX$
and recall assumption \cref{e:gap'}.
Then the \emph{shadow sequence} $(P_\bA\bT^n\bx)_\nnn$ converges to 
$\bar\bx =(\bar x, \dots, \bar x) \in \bA\cap (\bB -\bg)$, where 
$\bar x \in \bigcap_{j=1}^M (C_j -g_j)$ 
and $\bar x$ is a least-squares solution of \cref{e:moresets}.  
\end{corollary}
\begin{proof}
Combine \cref{t:affcvx} with 
\eqref{e:waterloo} and \cref{e:rela}.
\end{proof}
\begin{figure}[h]
\label{fig:many}
\begin{center}
\includegraphics[width=0.7\columnwidth]{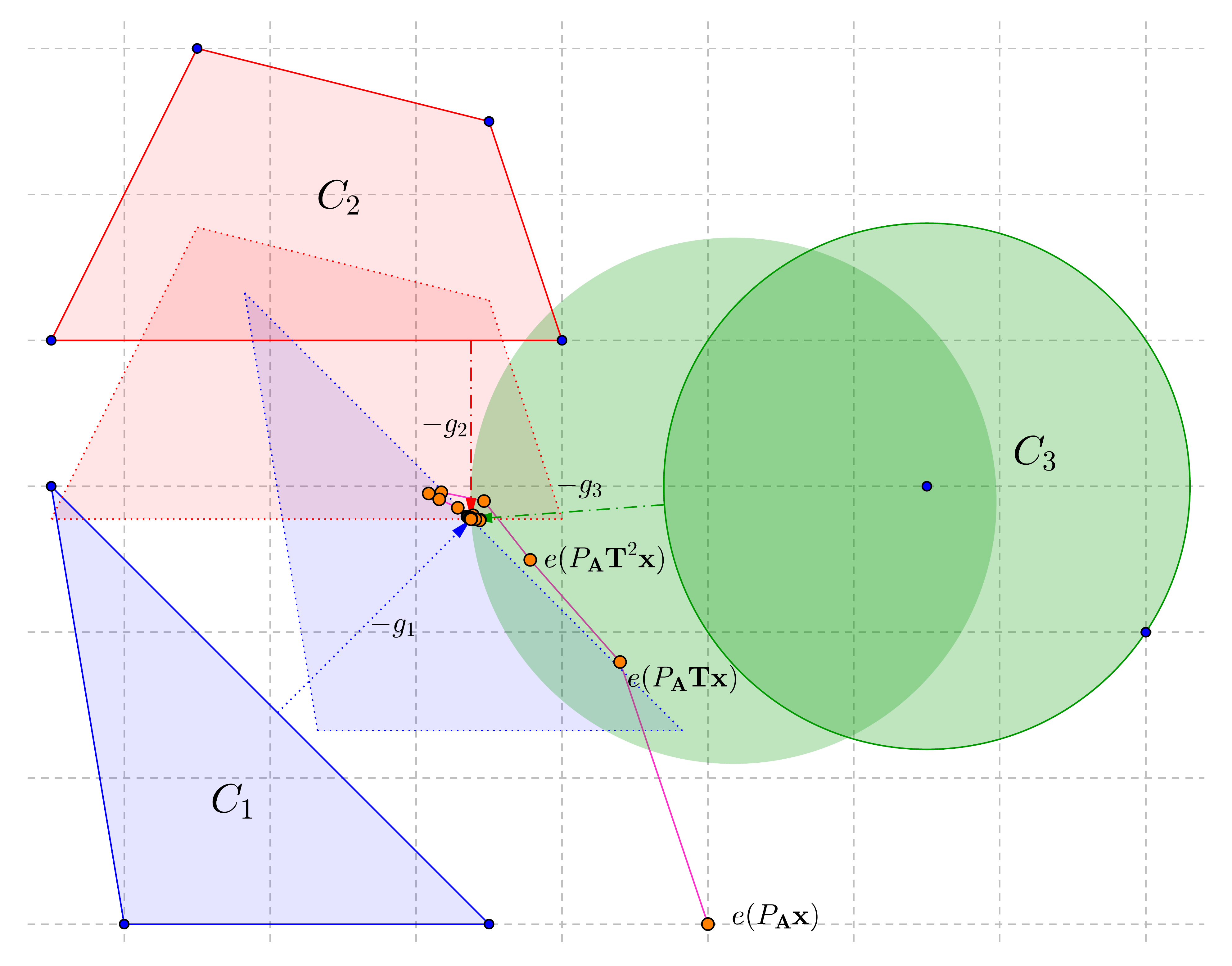}
\caption{A \texttt{GeoGebra} \cite{geogebra}
snapshot that illustrates \cref{c:moresets}.
Three nonintersecting
 closed convex sets, $C_1$ (the blue triangle),
 $C_2$ (the red polygon) and $C_3$ (the green circle), are shown 
along with their translations forming 
the generalized intersection. 
The first few terms of the 
 sequence $(e(P_\bA\bT^n\bx))_\nnn$ (yellow points)
 are also depicted. Here $e: \bA \to \RR^2: (x, x, x) \mapsto x$.}
\end{center}
\end{figure}

\begin{remark}
When we particularize \cref{c:moresets} from convex sets to 
\emph{halfspaces} and $X$ is \emph{finite-dimensional}, 
we recover 
Spingarn's \cite[Theorem~1]{Spi87}.
Note that in this case, in view of \cite[Facts~5.1(ii)]{BB94}
 we have $\bg\in \bB-\bA$. 
Recall that Spingarn used the following version of his
\emph{method of partial inverses} from \cite{Spi83}:
\begin{equation}
\label{e:MPI}
(\ba_0, \bb_0) \in \bA\times \bA^\perp 
\quad\text{and}\quad (\forall\nnn) \quad
\begin{cases}
\ba'_n =P_\bB(\ba_n +\bb_n), & \bb'_n =\ba_n +\bb_n -\ba'_n, \\
\ba_{n+1} =P_\bA\ba'_n, & \bb_{n+1} =\bb'_n -P_\bA\bb'_n.
\end{cases} 
\end{equation}
This method is the DRA in $\bX$, 
applied to $\bA$ and $\bB$ with starting point $(\ba_0 -\bb_0)$
(see, e.g., \cite[Lemma~2.17]{BDNP15}).
\end{remark}

\end{document}